\documentclass{amsart}
\usepackage[english]{babel}
\usepackage{esint}
\usepackage{amsfonts,amssymb,amsmath}
\usepackage{setspace}
\usepackage{centernot}
\usepackage{mathrsfs} 
\usepackage{pst-plot}
\usepackage{auto-pst-pdf}

\def \supcomn { C_0 ^{\infty} (\erre^{N})}
\newcommand{\erre} {{\mathbb {R}}}

\newcommand {\elle} {{\mathcal {L}}}
\def \effe {{\mathscr {F}}}

\def \acca{{\mathcal {H}}}

\def\erren{{\erre^{ {N} }}}
\def\erreu{{\erre^{ {N+1} }}}

\def\inn{\mbox{ in }}

\def\andd{ \quad\mbox{ and } \quad }

\newcommand{\tende}{\rightarrow}
\newcommand{\ttende}{\longrightarrow}
\newcommand{\enne} {\mathbb{N}}
\newcommand{\frecciaf} {\longmapsto}

\newcommand{\inte} {\cap}

\newcommand\partiali{\partial_{x_i}}
\newcommand\partialij{\partial_{x_i x_j}}

\newtheorem{theorem}{Theorem}[section]
\newtheorem{proposition}[theorem]{Proposition}

\theoremstyle{remark}

\theoremstyle{definition}

\numberwithin{equation}{section}

\title[On the Dirichlet Problem for hypoelliptic evolution equations]{On the Dirichlet Problem for hypoelliptic evolution equations: \\ Perron-Wiener solution \\and a cone-type criterion}
\author{Alessia E. Kogoj}
\address {Dipartimento di Ingegneria dell'Informazione, Ingegneria Elettrica e Matematica Applicata, Universit\`a degli Studi di Salerno, Via Giovanni Paolo II, 132, IT-84084 Fisciano (SA) - Italy}

\email{akogoj@unisa.it}


\begin{document}

\begin{abstract} 
We show how to apply harmonic spaces potential theory
in  the study of the Dirichlet problem for a general class of
evolution hypoelliptic partial differential equations of second order.
We construct Perron-Wiener solution and we provide a sufficient condition 
 for the regularity of the boundary points.
Our criterion extends and generalizes  the classical parabolic-cone
criterion for the Heat equation due to Effros and Kazdan.

\end{abstract}
\maketitle
\section{Introduction}
The aim of this paper is to prove the existence of a generalized solution in the sense of Perron-Wiener to the Dirichlet problem and to provide a sufficient condition for the regularity of the boundary points for a wide class of evolution equations. 

More precisely, we consider second order partial differential operators of the following type

\begin{equation} \label{operatore} 
  \elle = \sum_{i,j=1}^N a_{ij}(z) \partialij  + 
\sum_{i=1}^N b_i(z) \partiali  - \partial_t,
\end{equation}

in a strip 
$$S=\{ z=(x,t)\in\erreu \ |\ x\in\erren, \ T_1<t<T_2\},$$
with $-\infty\le T_1<T_2 \le + \infty.$

The coefficients $a_{ij}=a_{ji}$ and $b_i$ are smooth and the characteristic form of the operator is nonnegative definite and non totally degenerate, i.e., 
$$\sum_{i,j=1}^N a_{ij}(z) \xi_i\xi_j \geq 0, \qquad \forall z  \in S, \ \forall \xi=(\xi_1,\ldots,\xi_N) \in \erre^N,$$

and 
$$\sum_{i=1}^N a_{ii}(z) >0 \qquad \forall z \in S.$$ 

Finally, we assume the {\it hypoellipticity} of  $\elle - \beta$ and of $\elle^*$,  for every constant $\beta \ge 0$, and 
the existence of a well-behaved  {\it fundamental solution} $\Gamma$ for $\elle$, $$(z,\zeta) \frecciaf \Gamma(z,\zeta),$$
satisfying the following properties:

\begin{itemize}
\item[{\rm(i)}]
$\Gamma(\cdot,\zeta)$
belongs to  $L^1_{\rm loc}(S)$ and $\elle(\Gamma(\cdot,\zeta)) = - \delta_{\zeta}$, where  $\delta_{\zeta}$ denotes the Dirac measure at $\{\zeta\}$,  for every  $\zeta\in S$.

 \item[{\rm(ii)}] For every  $\varphi\in\supcomn$ and for every $(x_0,\tau)\in S$, 

\[ \int_{\erre^{N}}
  \Gamma (x,t, \xi,\tau )\varphi(\xi)\ d\xi \tende  \varphi(x_0), \quad \mbox{ as } (x,t) \tende (x_0,\tau),\ t>\tau.\]

\item[{\rm(iii)}] $\Gamma\in
C^{\infty}\Big(\{(z,\zeta)\in\erre^{N+1}\times \erre^{N+1}\ |\
z\neq\zeta\}\Big).$

\item[{\rm(iv)}]$\Gamma \geq 0$ and $\Gamma (x,t,\xi,\tau)>0$ if
and only if  $t>\tau$.  Moreover, for every fixed $z\in S$, $\limsup_{\zeta \tende z} \Gamma(z,\zeta) = \infty$.
\item[{\rm(v)}] $\Gamma(z,\zeta)\tende 0$ for $\zeta\tende\infty$ uniformly for $z\in K$, compact set of $S$, and, analogously, 
$\Gamma(z,\zeta)\tende 0$ for $z \tende\infty$ uniformly for $\zeta \in K$, compact set of $S$.

\item[{\rm(vi)}]  $\exists\, C>0$ such that for any  $z=(x,t)\in S$  we have $$\int_{\erre^N} \Gamma(z; \xi, \tau) \ d \xi \le C \mbox{\quad if\quad  } t>\tau.$$

\end{itemize}

We observe that the Kolmogorov-Fokker-Planck-type operators studied in \cite{garofalo_lanconelli_kolmogorov} and in \cite{lanconelli_polidoro_1994}, the 
ultraparabolic operators studied in \cite{kogoj_lanconelli_2004} and  in  \cite{link}, and the diffusion operators studied in \cite{BBLU} and \cite{lanconelli_uguzzoni_2010} belong to the class of operators considered in this paper.

The paper is organized as follows. Section 2  is devoted to briefly recall the notions and results from  Potential Theory that we need
to study the Dirichlet problem for our class of operators. 
In Section 3, following the techniques used by Lanconelli and Uguzzoni in \cite{lanconelli_uguzzoni_2010}, we show that the set of the solutions $u$ of $\elle u=0$ in $\Omega$  is a $\beta$-harmonic space satisfying the axioms of the Doob Potential Theory. In this way, in Section 4 we derive the existence of a {\it generalized solution in the sense of Perron-Wiener} in our setting.
We also recall a classical {\it balayage}-criterion that we will use to prove our main theorem.
In Section 5 we state and prove the main theorem: a criterion of boundary regularity which bases only on the behavior of the integral of the fundamental solution on a particular subset of $\erre^N$. As a consequence, in Section 6, we deduce  cone-type criteria inspired to the parabolic-cone 
criterion for the heat equation due to Effros and Kazdan \cite{effros_kazdan_1970} \cite{effros_kazdan_1971}.  Our criteria extend and generalize also the cone-type condition proved in \cite[Theorem 4.11]{lanconelli_uguzzoni_2010} (see also \cite{uguzzoni_cono}) for a class of hypoelliptic diffusion equations under the assumptions of doubling condition and segment property for an underlying distance and Gaussian bounds of the fundamental solution.
At the best of our knowledge, the only cone-type criterion for Kolmogorov-Fokker-Planck-type operators present in literature is related to the prototype of the Kolmogorov operator in $\erre^3$  and it is in the paper \cite{scornazzani_cono} where, for the same operator, Scornazzani proved a  Landis-Wiener-type criterion.  We would like to emphasize that,  in our general framework, i.e., for evolution equations with underline sub-Riemannian structures, the problem of characterizing the regularity of the boundary points in terms of Wiener-type series is still widely open. Nowadays, there are  only  few results in literature: the one related to the Kolmogorov equation in $\erre^3$ due to Scornazzani \cite{scornazzani_cono} and the Wiener criterion related to the heat operator on the Heisenberg group due to Garofalo and Segala \cite{garofalo_segala_wiener}. Very recently, for the operators studied in \cite{lanconelli_uguzzoni_2010}, Lanconelli, Tralli and Uguzzoni in \cite{ltu_2016} have given  necessary and sufficient regularity conditions in terms of Wiener-type series; however,  these criteria do not exactly characterize the boundary points.

\section{Potential Theory on Harmonic spaces: some recalls}

In this section, we recall some basic definitions and results
from the Potential Theory  that  will allow us to apply the Perron-Wiener
method to solve the Dirichlet  problem related to $\elle$.
We refer to refer to  \cite[pp. 22--23]{CC} for a historical note on the Perron-Wiener solution and to \cite[chapter 6]{BLU}, \cite{CC} and \cite{bauer}
for a detailed description of the general theory of {\it harmonic spaces}. 

Throughout this Section  $(X, \mathcal{T})$ will denote a topological Hausdorff space,
locally connected and locally compact. We also assume the topology $\mathcal{T}$ has a
countable basis. 

\subsection{Sheafs of functions and harmonic sheafs in $X$} 
\mbox{}\\
Let $V$ be any open subset of $X$. We denote by $\overline\erre$ the set $\mbox{$\erre \cup \{\infty, -\infty\}$ 
and by $\overline\erre^V$}$ the set of functions $\mbox{$u: V \longrightarrow \overline\erre$}$.  Moreover $C(V,\erre)$  is the vector space of real continuous functions defined on $V$. 
A map 
$$\effe : \mathcal{T} \ttende \bigcup_{V\in\mathcal{T}} \overline\erre^V$$ 

is a {\it sheaf of functions} in $X$ if
\begin{itemize}
\item[$(i)$] $\effe(V)\subseteq  \overline\erre^V \quad \forall\ V\in \mathcal{T};$
\item[$(ii)$] $V_1, V_2 \in \mathcal{T}$, $V_1 \subseteq V_2$, $ u\in \effe(V_2)$ $\implies$ $u|_{V_1} \in 
\effe({V_1})$;

\item[$(iii)$] $V_\alpha   \in \mathcal{T}\  \forall \alpha \in \mathcal{A},  u: \bigcup_{\alpha \in A} V_\alpha 
\longrightarrow \overline\erre$, $u|_{V_\alpha} \in \effe({V_\alpha}) \implies u\in \effe({ \bigcup_{\alpha\in 
\mathcal{A}}}V_\alpha).$ 

\end{itemize}

\noindent When $\effe(V)$ is a linear subspace of $C(V,\erre)$ for every $V\subseteq X$, we say that the sheaf of 
functions $\effe$ on $V$ is {\it harmonic} and we denote it by  $\acca(X).$ The functions belonging to $\acca(X)$ will be called {\it harmonic functions}.
\subsection{Regular open sets and harmonic measures} 
\mbox{}\\
Let $\acca$ be a harmonic sheaf on $X$. We say that a bounded open set $V\subseteq X$ is
{\it $\acca$-regular}
if:
\begin{itemize}

\item[$(i)$] $\overline{V}\subseteq X$ is compact and $\partial V \neq \emptyset$;

\item[$(ii)$] for every continuous function $\varphi : \partial V \longrightarrow \erre$, there exists a unique
function, $h_\varphi^V$,
in $\acca(V)$ and continuous in $\overline V$, such that  
$$h_\varphi^V|_{\partial V}=\varphi.$$
\item[$(iii)$] if $\varphi\geq 0$ then $h_\varphi^V \geq 0.$ 

\end{itemize}
From $(ii)$ and $(iii)$ it follows that, for every  regular set $V$ and for every $x \in V$, the map 
$$C(\partial V) \ni \varphi \longmapsto h_\varphi^V(x) \in \erre$$ is linear, continuous  and non-negative.
Thus, the Riesz representation theorem, 
implies that, for every  regular set $V$ and for every $x \in V$, there exists a {\it regular Borel measure}, that we 
denote by $\mu^V_x$, supported in $\partial V$, 
such that 
$$h_\varphi^V (x) = \int_{\partial V} \varphi(y) \ d \mu^V_x(y) \qquad \forall\ \varphi\in C(\partial V).$$ 
The measure $\mu^V_x$ is called the {\it $\acca$-harmonic measure} related to $V$ and $x$.
\subsection{Iperharmonic functions, Superharmonic functions, Potentials} \mbox{}\\
A function $u: X \ttende ]-\infty, \infty]$ is called {\it  $\acca$-iperharmonic in $X$} if 
 \begin{itemize}

\item[$(i)$] $u$ is lower semi-continuous;

\item[$(ii)$] for every regular set $V$, $\overline V\subseteq X$, and  for every $\varphi \in C( \partial V, 
\erre)$, 
$\varphi \le u|_{\partial V} $,  it follows  $ u \geq h_\varphi^V$ in $V;$
\end{itemize}
If $u$ is iperharmonic and the set $\{ x\in X \ | \ u(x) < \infty \}$ is dense in $X$, then $u$ is called {\it superharmonic}.

We will denote by  $\acca^*(X)$  the family  of the iperharmonic functions on $X$ and by $\overline{\acca}(X)$  the family  of the superharmonic functions.

A  {\it $\acca$-potential} in $X$ is a nonnegative superharmonic function such that any nonnegative harmonic minorant is identically zero.

\subsection{Doob $\beta$-harmonic spaces} \label{axioms}
\mbox{}\\
We say that a harmonic sheaf $\acca(X)$ is a {\it $\beta$-harmonic space satisfying the Doob convergence property}  if it verifies the following axioms.

 \begin{itemize}

\item[(A1)]  { \it Positivity axiom:}\\
For every $x\in X$, there exists a open set $V\ni z$ and a function $u\in \acca(V)$  
such that $u(x)>0$.

\item[(A2)]  { \it Doob convergence axiom:}\\
The limit of any increasing sequence of $\acca$-harmonic functions in a open set $V\subseteq X$ is $\acca$-harmonic whenever it is finite in a dense subset of $\Omega$. 

\item[(A3)]  { \it Regularity axiom:}\\   There is a basis of the euclidean topology of $X$ formed by   $\acca$-regular 
sets.

\item[(A4)]  { \it Separation axiom:}\\  
For every $y$ and $z$ in $X$, $y\neq z$, there exist two  $\acca$-potentials $u$ and $v$ in $X$ such that 
$u(y)v(z)\neq u(z)v(y).$
\end{itemize}

\subsection{Dirichlet problem in harmonic space}\mbox{}\\
Let $\Omega$ be an open set of $X$, with compact closure and non-empty boundary, and $\varphi: \partial\Omega\ttende \overline\erre.$ 
We call {\it generalized Dirichlet problem for the harmonic sheaf $\acca$ in the open set $\Omega$ with boundary data $\varphi$}, the problem of finding a function $u\in \acca(\Omega)$ such that 
$$\lim_{x\tende y} u (x)=\varphi(y)\quad\forall\ y\in\partial\Omega.$$

In this case we say that $u$ solves the problem 
\begin{equation} \label{pd}\tag{$\acca$-D}
\begin{cases} u\in  \acca(\Omega)   \\
  u|_{\partial \Omega} = \varphi.\end{cases}
\end{equation}
If $\varphi\in C(\partial\Omega)$ (and we are in a Doob $\beta$-harmonic space), the function 
$$H_\varphi^\Omega:= \inf\{ u\in {\acca^*} (\Omega) \ | \ \liminf_{z\tende \zeta} u(z) \geq \varphi(\zeta)\quad \forall\ \zeta\in \partial\Omega\}$$
belongs to the harmonic sheaf $\acca(\Omega)$ (see \cite[Theorem 2.4.2]{CC}) and it is called the {\it generalized solution in the sense of Perron-Wiener} to the Dirichlet problem ($\acca$-D). 

A point $z_0\in\partial\Omega$ is called {\it $\acca$-regular for $\Omega$} if 
 $$\lim_{z\tende z_0} H_\varphi^\Omega(z)=\varphi(z_0)\quad\forall\ \varphi\in C(\partial\Omega).$$

Of course, if (and only if) every point of $\partial\Omega$ is $\acca$-regular the function $H_\varphi^\Omega$ is the (unique) solution  to 
\begin{equation*} 
\begin{cases} u\in  \acca(\Omega)\inte C(\overline\Omega)   \\
  u|_{\partial \Omega} = \varphi\end{cases}
\end{equation*}
for every $\varphi(\partial \Omega).$

\section{The harmonic space of the solution of $\elle u=0.$}\label{potentialelle}
In this section, we show that the set of the solutions of the equation $\elle u=0$ is  $\beta$-harmonic space in $S$  satisfying the Doob convergence property.

For every open set $\Omega\subseteq S$ we set
\begin{equation*} \acca(\Omega): = \{ u\in C^\infty(\Omega) \ | \ \elle u=0 \},
\end{equation*}
where  $\elle$ is the operator  \eqref{operatore}.   Then, 
$$ \Omega \frecciaf \acca(\Omega)$$ is a  {\it harmonic sheaf of functions}  in $S$.

The assumptions on $\elle$ and on its fundamental solution allow us to prove the following theorem. 
\begin{theorem}\label{dueuno} Let $S'=\erre^N\times ]T'_1,T'_2[$ be a strip of $\erre^{N+1}$ where $T_1<T'_1<T'_2<T_2$. Then $\acca (S')$ is a
Doob  $\beta$-harmonic space.
\end{theorem} 

\begin{proof} We follow verbatim the lines of the proof of Theorem 3.9 in \cite{lanconelli_uguzzoni_2010}. 
Here, for the convenience of the reader, we repeat the main points (referring to  \cite[Section 3]{lanconelli_uguzzoni_2010} for their proofs). 
Let us start recalling a Minimum Principle for $\elle$ (see \cite[Proposition 3.1]{lanconelli_uguzzoni_2010}) .
\begin{proposition} \label{duedue} Let $\Omega$ be an open set, $\overline\Omega \subseteq S.$ For any $T\in ]T_1,T_2[$ we set
$$\Omega_T = \Omega \inte \{(x,t)\ | \ t<T\} \andd \partial_T \Omega = \partial \Omega \inte \{(x,t)\ | \ t\le T\}.  $$
Let $u$ be a $C^2$ function in $\Omega$ such that 
\begin{itemize}
\item[$(i)$] $\elle u \le 0$ in $\Omega$;
\item[$(ii)$] $\liminf_{\Omega_T \ni z\tende \zeta} u(z) \geq 0$ for every $\zeta \in \partial_T\Omega$;
\item[$(iii)$] $\liminf_{\Omega_T \ni z\tende \infty} u(z) \geq 0$ if $\Omega_T$ is not bounded. 
\end{itemize} 
Then, $u\geq 0$ in $\Omega_T$.
 \end{proposition} 
 As a consequence, 
for every $V\subseteq S$ $\elle$-regular, the (unique) function $h_\varphi^V$
in $\acca(V)$, continuous in $\overline V$ and such that $h_\varphi^V|_{\partial V}=\varphi$,  is non-negative if  $\varphi\geq 0$. 
Therefore, for every  regular set $V$ and for every $x \in V$, the map 
$$C(\partial V) \ni \varphi \longmapsto h_\varphi^V(x) \in \erre$$ is linear, continuous  and non-negative functional, and 
we can write 
$$h_\varphi^V (x) = \int_{\partial V} \varphi(y) \ d \mu^V_x(y) \qquad \forall\ \varphi\in C(\partial V),$$ 

where $\mu^V_x$ is the harmonic measure related to $V$ and $x$. Now, from the Minimum Principle,  the hypoellipticity, the non totally degeneracy of the operator $\elle$, 
making use of a standard argument (see \cite[Corollarie 5.2]{bony_1969}, see also \cite[Proposition 7.1.5]{BLU}), it follows that the family of the $\elle$-regular sets
$$\{V\subseteq S\ | \ \mbox{$V$ open and $\elle$-regular} \}$$ 
is a basis of the euclidean topology of $S$, thus the {\it regularity axiom} is satisfied.

The  { \it Doob convergence axiom}  is a consequence of a weak Harnack inequality  due to Bony (see \cite[Theoreme 7.1]{bony_1969}); see also \cite[Proposition 7.4]{kogoj_lanconelli_2004}).

The {\it positivity axiom} (A1) is plainly verified. Indeed every constant function belongs to $\acca(\Omega)$.

We are left to prove that the {\it separation axiom } (A4) holds in our setting.

For every fixed $\zeta_0\in S$ the function 
$$z\frecciaf \Gamma(z,\zeta_0)  \mbox{ is a }  \acca\mbox{-potential}.$$ Indeed  $\Gamma$ is nonnegative and $\acca$-superharmonic in $S$.  Moreover, if $h\in\acca(S)$ and $h\le\Gamma(\cdot,\zeta_0)$, then $h\le 0$ in $S$ (see \cite[Proposition 3.4]{lanconelli_uguzzoni_2010}). 
This result, together with property $(iv)$ of $\Gamma$, allows us to verify the separation axiom:

For every $z_1$ and $z_2$ in $S'$, $z_1\neq z_2$, there exist two  $\acca$-potentials $u$ and $v$ such that 
$$u(z_1)v(z_2)\neq u(z_2)v(z_1).$$

Thanks to property $(iv)$, we can find a sequence $(\zeta_j)$ such that  $\zeta_j \ttende z_1$ such that $$\Gamma(z_1,\zeta_j)\ttende \infty\mbox{ for }  j\ttende \infty,$$ 
where $\zeta_j=(\xi_j,\tau_j)$ with $\tau_j<t_1.$ Now, we set 
$$u_j=\Gamma(\cdot, \tau_j).$$
$u_j$ is a $\acca$-potential and, for every $j\in \enne$, 
\begin{equation*} \begin{split} \lim_{k\tende\infty}& \left(u_j(z_1)u_k(z_2) - u_j(z_2)u_k(z_1)\right) \\
& = \Gamma(z_1,\zeta_j)\Gamma(z_2,z_1) - \Gamma(z_2,z_j) \lim_{k\tende\infty} \Gamma(z_1,\zeta_k) \\
&=  - \infty.
 \end{split}\end{equation*}
 Hence, there exist $j,k\in\enne$ such that 
 $$u(z_1)v(z_2)\neq u(z_2)v(z_1),$$
 and the proof is complete.
\end{proof}

From Theorem \ref{dueuno} and the theory of harmonic space, we obtain an extension of Proposition \ref{duedue}
(see \cite[Proposition 3.10]{lanconelli_uguzzoni_2010}).

\begin{proposition} \label{duequattro} Let $\Omega$ be an open set, $\overline\Omega \subseteq S,$
$$\Omega_T = \Omega \inte \{(x,t)\ | \ t<T\} \mbox{ and }\   \partial_T \Omega = \partial \Omega \inte \{(x,t)\ | \ t\le T\}, \  \mbox{for any }   T\in ]T_1,T_2[.$$

Let $u$ be a superharmonic function in  $\overline{\acca}(\Omega_T)$, $T\in ]T_1,T_2[$, such that 
\begin{itemize}
\item[$(i)$] $\liminf_{\Omega_T \ni z\tende \zeta} u(z) \geq 0$ for every $\zeta \in \partial_T\Omega$;
\item[$(ii)$] $\liminf_{\Omega_T \ni z\tende \infty} u(z) \geq 0$ if $\Omega_T$ is not bounded. 
\end{itemize} 
Then, $u\geq 0$ in $\Omega_T$.
 \end{proposition}

\section{The Dirichlet problem for $\elle$}
\subsection{The Perron-Wiener solution}

Let $\Omega$ be a bounded open set  with $\overline \Omega \subseteq S$ and $\varphi \in C(\partial \Omega).$ We consider the Dirichlet problem
\begin{equation} \label{DP}\tag{DP}
\begin{cases}
 \elle   u= 0 \mbox{ in } \Omega    \\
  u|_{\partial \Omega} = \varphi \end{cases}
\end{equation}

Since the operator $\elle$ endows the strip $S$ with a structure of Doob $\beta$-harmonic space, by the Wiener resolutivity theorem 
 we always have the existence of a {\it generalized solution in the sense of Perron-Wiener} to the Dirichlet problem (DP) 
$$H_\varphi^\Omega:= \inf\{ u\in {\acca^*} (\Omega) \ | \ \liminf_{z\tende \zeta} u(z) \geq \varphi(\zeta)\quad \forall\ \zeta\in \partial\Omega\}.$$

$H_\varphi^\Omega$ is $C^\infty(\Omega)$ and satisfies $\elle u=0$ in $\Omega$. When the Dirichlet problem \eqref{DP} has a solution $u$ in the classical sense, it will turn out that $u=H_\varphi^\Omega.$

Viceversa, if 
$$\lim_{x\tende y} H_\varphi^\Omega(x)=\varphi(y)\quad\forall\ y\in\partial\Omega,$$ $u \in \acca(\Omega)\cap C(\overline\Omega)$ and solves the problem \eqref{DP} in classic sense. However,  in general, $H_\varphi^\Omega$ does not assume the datum $\varphi$ on $\Omega$. In the next sections we are going to give some conditions of boundary regularity  for $\elle$.

We will use a classical criterion from potential theory that characterizes the regularity of boundary points in term of the balayage on the complementary of $\Omega$.

\subsection{Balayage and a regularity criterion}
Given a compact set $K\subseteq S$, let $W_K$ and $V_K$ be, respectively, the {\it reduced function} and the {\it balayage} of $1$ on $K$:

$$W_K:= \inf \{ v\ | v \in {\acca^*} (S) ,\ v\geq 0 \inn S,\  v\geq 1 \inn K\}$$
and 
$$V_K(z)=\liminf_{\zeta\ttende z} W_K(\zeta),\qquad z\in S.$$

From general balayage theory we have that 
$V_K$ is equal to $1$ on the interior of $K$, vanishes at infinity, is a superharmonic function on S and harmonic on $S\backslash \partial K$ (see \cite[Proposition 4.1]{lanconelli_uguzzoni_2010}). Moreover we can characterize the regularity of the boundary point of an open set $\Omega$ by the following condition (see \cite[Proposition 4.6]{lanconelli_uguzzoni_2010}).

\begin{proposition}\label{regularpoint} 
Let $\Omega$ be a bounded open set with $\overline\Omega \subset S$ and let $z_0$ be a point of $\partial \Omega.$ Let 
$(B_\lambda)_{0<\lambda<1}$ be a basis of closed neighborhood of $x_0$ (in $\erre^N$) such that $B_\lambda\subseteq B_\mu$ if $0<\lambda<\mu\le1$.  For every $\lambda$,we set 

$$\Omega^c_\lambda(z_0):= (B_\lambda \times [t_0 - \lambda, t_0])\backslash \Omega.$$
Then, $z_0\in\partial\Omega$ is $\elle$-regular if and only if 
$$\lim_{r\tende 0} V_{\Omega^c_r (z_0)} (z_0) > 0.$$\end{proposition}

\section{Main Theorem}
Let $\Omega$ be a bounded open set with $\overline\Omega \subset S$ and let $z_0$ be a point of $\partial \Omega.$ We define $\Omega^c_\lambda(z_0)$ as in Proposition \ref{regularpoint}  and we denote 
 $$T_\lambda(z_0)=\{ x\in\erre^N\ : \ (x,t_0-\lambda)\in \Omega^c_\lambda(z_0)\}.$$
\vspace{0.5cm}

\begin{flushright}

\includegraphics[width=.6\textwidth]{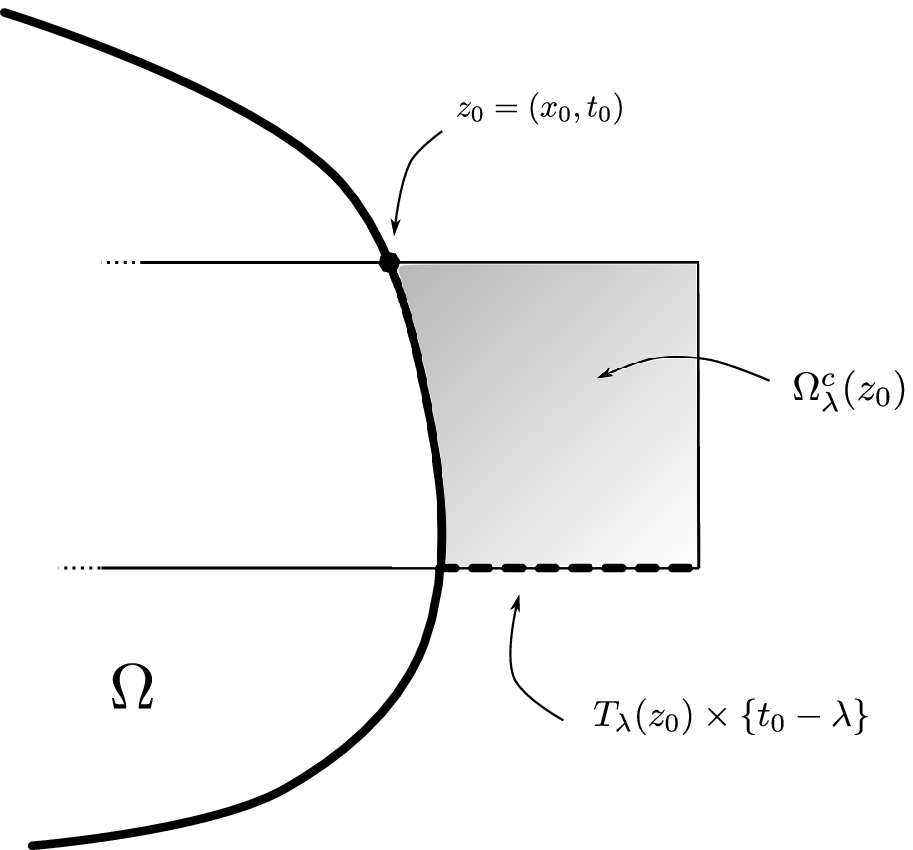}
\end{flushright}

Finally, we define 

$$\gamma_\lambda(z_0) = \int_{T_\lambda(z_0)} \Gamma (z_0; \xi, t_0 - \lambda)\ d\xi$$
and we state our main theorem.

\begin{theorem}\label{mainth}
The point $z_0\in\partial\Omega$ is $\elle$-regular if 
\begin{equation}\label{3.1} \limsup_{\lambda\searrow 0} \gamma_\lambda(z_0) > 0.\end{equation} \end{theorem}
\begin{proof}
We know that $z_0$ is $\elle$-regular if and only if 
\begin{equation}\label{3.2} \lim_{\lambda \tende 0} V_{\Omega^c_\lambda (z_0)} (z_0) > 0\end{equation}
(see Section \ref{potentialelle}, Proposition \ref{regularpoint}). Therefore we have to prove condition \eqref{3.2}. To this end, using 
\eqref{3.1}, we choose a sequence $\lambda_n\searrow 0, \lambda_n>0,$ and a constant $a>0$ such that 
$$\gamma_{\lambda_n}(z_0)\geq a\qquad \forall\ n\in\enne.$$
Let us now consider the strip 
$$S_n=\erre^N\times ]t_0-\lambda_n, t_0[,\quad n\in\enne,$$ 
and define 
$$v_n(z) = \int_{T_{\lambda_n}(z_0)} \Gamma (z; \xi, t_0 - \lambda_n)\ d\xi.$$
We will prove the inequalities 
\begin{equation}\label{tretre} V_{\Omega^c_{\lambda_n} (z_0)} \geq v_n \inn S_n,\ \forall n\in \enne.\end{equation}
As a consequence, we will have 
$$\lim_{n\tende\infty}  V_{\Omega^c_{\lambda_n} (z_0)} (z_0)\geq \limsup_{n\tende \infty} v_n(z_0)=  \limsup_{n\tende \infty} \gamma_{\lambda_n}(z_0)\geq a >0,$$
which implies \eqref{3.2}.

To prove \eqref{tretre} we first remark that $v_n$ is $\acca$-harmonic in $S_n$ and that 
$$\lim_{\substack{{z\tende \infty}\\{z\in S_n}}} v_n(z)=0.$$ 
Moreover $$v_n(z)\le \int_{\erre^N} \Gamma (z; \xi,t_0-\lambda_n)\ d\xi = 1,$$ for every $z\in S_n$, and 
$$\lim_{S_n\ni z\tende (y_0,t_0-\lambda_n)} v_n(z)=0\qquad\forall y_0 \notin T_{\lambda_{n}}(z_0).$$
Now, let $w\in \overline{\acca}(S_n)$, $w\geq 0$ in $S_n$ and $w\geq 1$ in $\Omega_{\lambda_n}^c(z_0).$

Then,
$w-v_n\in \overline{\acca}(S_n)$ and: \begin{equation*}
\liminf_{S_n\ni z\tende (y,t_0-\lambda_n)} (w(z)- v_n(z)) \geq w((y,t_0-\lambda_n)-1\geq 0,\end{equation*}
for every $y \in T_{\lambda_n}(z_0),$ and
$$\liminf_{S_n\ni z\tende (y,t_0-\lambda_n)} (w(z)- v_n(z)) \geq  
\liminf_{S_n\ni z\tende (y,t_0-\lambda_n)} w(z) \geq 0,$$
for every $y \notin T_{\lambda_n}(z_0).$

Finally,
$$\liminf_{S_n\ni z\tende (y,t_0-\lambda_n)} (w(z)- v_n(z))\geq 0,\qquad \forall\ y \in\erre^N.$$

Then, by the Minimum Principle for $\acca$-superharmonic functions (see Proposition \eqref{duequattro}), we get 
$$w-v_n\geq 0 \inn S_n.$$
Taking the infimum with respect to $w$ in this inequality we obtain \eqref{tretre}, completing the proof. 

\end{proof}

\section{Applications: cone-type criteria for evolution equations}
In this section we prove cone-type criteria  for two classes of evolution equations. 
\subsection{Invariant and homogeneous operators on a group in $\erre^{N+1}$}
We consider operators left translation invariant and homogeneous of degree two with respect an homogeneous group 
\begin{equation} \mathbb{G}=(\erre^{N+1}, \circ, \delta_r).\end{equation}
Notions and results about homogeneous groups can be found in the first chapter of the monograph \cite{BLU}. 
For the reader convenience, we shortly recall the definition of {\it homogeneous group in $\erre^{N+1}$} adapted to our setting. 
The triple $\mathbb{G}$ in \eqref{3.1} is called homogeneous Lie group if  $(\erre^{N+1},\circ)$ is a
Lie group and if $(\delta_r)_{r >0}$ is a group of {\it homomorphisms} on $(\erre^{N+1},\circ)$ of the following type
\begin{eqnarray*}
\delta_r:\erre^{N+1}\ttende\erre^{N+1},\quad \delta_r (x_1,\ldots, x_N,t)= (r^{\sigma_1} x_1,\ldots,
r^{\sigma_N} x_N, r^2 t ),\end{eqnarray*}
where $\sigma_1, \ldots, \sigma_p$ are positive integers such that $1 \le \sigma_1 \le ... \le  \sigma_N$.

We set $$D_r = \delta_r|_{\erre^N}.$$  $(D_r)_{r>0}$ is group of dilation in $\erre^N$. The natural number 
$$Q= \sigma_1 + ... +  \sigma_N$$ is the {\it homogeneous dimension of $\erre^N$ with respect to $(D_r)_{r>0}$}
while $$Q+2$$ is the {\it homogeneous dimension of $\erre^{N+1}$ with respect to $(\delta_r)_{r>0}$}.

We suppose $\elle$ to be left translation invariant on $(\erre^{N+1}, \circ)$ and homogeneous of degree two with respect the dilations 
$(\delta_r)_{r>0}$. Denoting $k$ the fundamental solution of $\elle$ with pole at the origin $(0,0)$, the last hypothesis we need  is that the fundamental solution of $\elle$ satisfies the following properties:

\begin{itemize}
\item[(a)] $\Gamma(z,\zeta) =k(\zeta^{-1}\circ z)$; 
\item[(b)]  $k(\delta_r(z))= r^{-Q} k(z).$
\end{itemize} 
Operators belonging to this class are, for example, the heat operators on stratifies Lie groups, the ultraparabolic operators introduced and studied in  \cite{kogoj_lanconelli_2004, link} and the {\it homogeneous} prototypes of Kolmogorov-Fokker-Planck operators studied in \cite{lanconelli_polidoro_1994}.
 
We name {\it $\delta_r$-cone with vertex in $(0,0)$} every open set of the following kind:
\begin{equation*} \hat C:=\{ \delta_r(\xi, -T)\ | \ \xi\in B,\ 0<r <1\} , = \{ (D_r(\xi), -r^2 T) \ |\ \xi\in B,\ 0<r <1 \},
\end{equation*}
where $T>0$ and $B$ is a bounded open set of $\erre^N$, $\mathrm{int} B\neq\emptyset.$ 

We name  {\it $\delta_r$-cone with vertex in $z_0$}  the set $$z_0\circ \hat C,$$ where $\hat C$ is a  $\delta_r$-cone with vertex in $0$.

\begin{flushright}
\includegraphics[width=.6\textwidth]{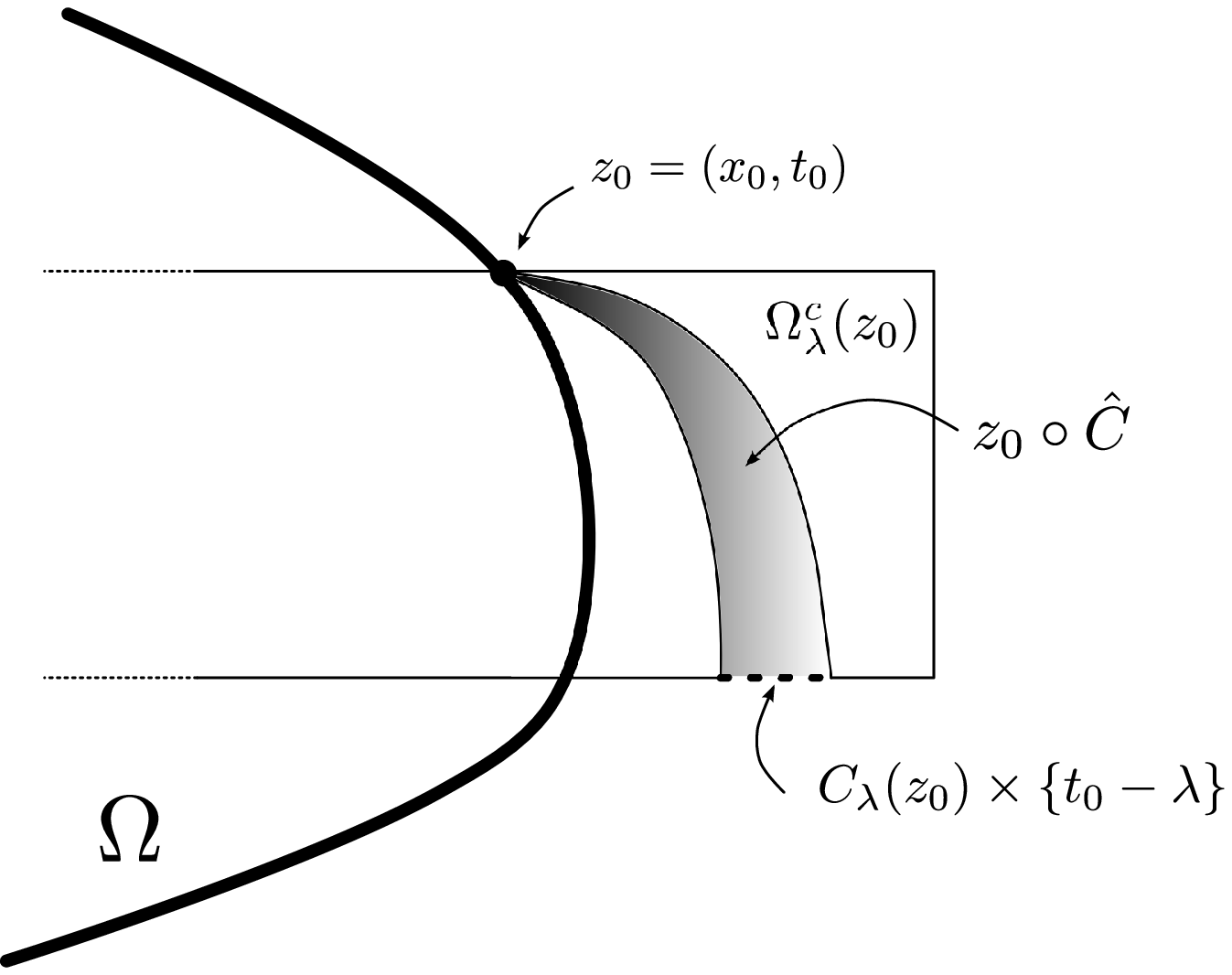}

\end{flushright}

Thanks to our main theorem (Theorem \ref{mainth}), we derive now the following cone-type criterion that extends the {\it parabolic cone (or tusk) condition} by Effros and Kazdan.
\begin{theorem} \label{criterio1} Let $\elle$ be an invariant evolution operator on $\mathbb{G}=(\erre^{N+1}, \circ, \delta_r).$ Let $\Omega$ be a bounded open set of $\erre^{N+1}$ and let be $z_0 \in \partial\Omega.$ If there exists a $\delta_r$-cone with vertex in $z_0$ contained in $\erre^{N+1}\backslash \Omega$, then $z_0$ is $\elle$-regular for $\Omega.$ 
\end{theorem} 

\begin{proof} 
As the operator $\elle$ is left translation invariant on $\mathbb{G}$, it is sufficient to prove the theorem in the case $z_0=(0,0).$ So, let $\hat C$ a $\delta_r$-cone with vertex in $(0,0)$ such that $\hat C \subseteq \erre^{N+1}\backslash \Omega.$  
Let $\lambda$ be in  $]0,T[,$ we set 
$$C_\lambda (0)= \{ x\in \erre^N \ | \ (x,-\lambda) \in \hat{C}  \}.$$
For any $W$ neighborhood of $0$ (in $\erre^N$), there exists $\lambda_0\in ]0,T[$  such that  $C_\lambda (0)\subseteq W$ for every $\lambda\in  ]0,\lambda_0[.$ We observe that $$C_\lambda (0)= D_r (B) \mbox{\quad  with \quad }r= \sqrt{\frac{\lambda} {T}}.$$ 
In particular 
\begin{equation} \label{lllambda}|C_\lambda (0)| = |D_r (B)| = r^Q |B|.\end{equation} 
Moreover, if $z\in C_\lambda(0)\times \{-\lambda\}$ and $r= \sqrt{\dfrac{\lambda} {T}}$, then

\begin{equation*}\begin{split}\Gamma(0,z)&=k(z^{-1})= k((\delta_r(\xi,-T))^{-1})=r^{-Q} k((\xi,-T)^{-1})\\
&\le r^{-Q} \min_{\xi\in\overline{B}} k((\xi,-T)^{-1}) = r^{-Q}  a_0. \end{split}
\end{equation*}
Using \eqref{lllambda}, we get,
\begin{equation*} \Gamma(0,z) \geq  \frac{a_0}{r^{Q}}= \frac{a_0 |B|}{|C_\lambda(0)|} = \frac{a}{|C_\lambda (0)|}. \end{equation*}
So, 
\begin{equation*} \int_{C_\lambda(0)} \Gamma(0; \xi,-\lambda)\ d\xi  \geq  \frac{a}{|C_\lambda (0)|}  |C_\lambda (0)| = a\qquad \forall \lambda \in ] 0,\lambda_0[.  \end{equation*}
Now, an application of Theorem \eqref{mainth} proves that $0$ is $\elle$-regular.

\end{proof}

\subsection{Kolmogorov-Fokker-Planck-type operators}
We formulate now a cone-type criterion for the class of operators introduced in \cite{lanconelli_polidoro_1994}  and subsequently studied by many authors as a basic model for general Kolmogorov-Fokker-Planck operators.
More precisely, we consider the operators 
in $\mathbb{R}^{N+1 }$
\begin{eqnarray}\label{KFP}
\elle = \mathrm{div}\left(  A\nabla\right)  +\left\langle B x,\nabla\right\rangle - \partial_t
\end{eqnarray}
where $A = (a_{i,j})_{i,j = 1, \dots , N}$ and $B= (b_{i,j})_{i,j = 1, \dots , N}$ are constant $N\times N$ matrices, $A$ is symmetric and nonnegative definite. in \cite{lanconelli_polidoro_1994}, it is proved that if we define the matrix
\begin{equation}\label{C(t)}
C(t)=\int_{0}^{t} E(s)AE^T(s)\,ds\mbox{, where }E(s) =\exp\left(-sB\right),  %
\end{equation}
 the operator $\mathcal{H}$ is hypoelliptic if and only if  $C(t)>0$ for every $t>0$. Furthermore, under these conditions, for some basis
of $\mathbb{R}^{N}$, the matrices $A,B$ take the following  form:%
\begin{equation*}
A=%
\begin{bmatrix}
A_{0} & 0\\
0 & 0
\end{bmatrix} \label{A}%
\end{equation*}
for some $p_{0}\times p_{0}$ symmetric and positive definite
constant matrix $A_{0}$ ($p_{0}\leq N$), and
\begin{equation*}
B=%
\begin{bmatrix}
\ast & \ast & \ldots & \ast& \ast\\
B_{1}  & \ast & \ldots & \ast & \ast \\
0 & B_{2}& \ldots & \ast &\ast \\
\vdots & \vdots & \ddots & \vdots & \vdots\\
0 & 0 & \ldots &  B_{r} & \ast
\end{bmatrix},
\label{B}%
\end{equation*}
where $B_{j}$ is a $p_{j-1}\times p_{j}$ block with rank $p_{j}$ ($j=1,2,...,r$), $p_{0}\geq p_{1}\geq...\geq p_{r}\geq1$ and $p_{0}+p_{1}+...+p_{r}=N$.

The operator $\elle$ is left-invariant with respect to the Lie group $\mathbb {K}$ whose underlying manifold is $\mathbb{R}^{N+1}$, endowed with the composition law
$$
\left(  x,t\right)  \circ\left(  \xi,\lambda\right)  =\left(  \xi+E(\lambda)  x,t+\lambda\right).
$$ 
Under the assumptions stated above, the operator $\elle$ in (\ref{KFP}) has a fundamental solution
\[
\Gamma\left(  z,\zeta\right)  =k\left(  \zeta^{-1}\circ z\right)  \text{
for }z,\zeta\in\mathbb{R}^{N+1},
\]
with%
\begin{equation}
k\left(x,t\right)  =\left\{
\begin{tabular}
[c]{ll}%
$0$ & $\text{for }t\leq0$\\
$\frac{\left(  4\pi\right)  ^{-N/2}}{\sqrt{\det C\left(  t\right)  }}%
\exp\left(  -\frac{1}{4}\left\langle C^{-1}\left(  t\right)  x,x\right\rangle
\right)  $ & $\text{for }t>0$%
\end{tabular}
\ \right. \nonumber
\end{equation}
where $C\left(  t\right)  $ is as in (\ref{C(t)}). Recall that $C\left(  t\right)  $ is positive definite for all
$t>0$; hence $k \in C^{\infty}\left(  \mathbb{R}^{N+1}\backslash\left\{
0\right\}  \right)$. Furthermore, $\Gamma$ satisfies condition $(i)-(vi)$ in Section 1. 

Let us now consider, for every $r>0$, the dilations

\begin{eqnarray*}\delta_r:\erre^{N+1}\ttende\erre^{N+1},\quad \delta_r(x,t)&=&
\delta_r(x^{(p)},x^{(p_1)},\ldots,x^{(p_k)},t)\\&=&(r
x^{(p_0)},r^3
x^{(p_1)},\ldots,p^{2k+1}x^{(p_k)},r^2t)\\
x^{(p_j)}\in\erre^{p_i},\quad j=0,\ldots,k,\quad r>0.\end{eqnarray*}
We wish to explicitly recall that $\delta_r$ is an authomorphism of $\mathbb{K}$ if and only if the all the blocks $\ast$
in $B$ are identically zero.

As in the previous subsection we call {\it $\delta_r$-cone with vertex in $(0,0)$} any open set of the  kind:
\begin{equation*} \hat C:=\{ \delta_r(\xi, -T)\ | \ \xi\in B,\ 0<r <1\} , 
\end{equation*}
where $T>0$ and $B$ is a bounded open set of $\erre^N.$

We name  {\it $\delta_r$-cone with vertex in $z_0$}  every set $$z_0\circ \hat C,$$ where $\hat C$ is a  $\delta_r$-cone with vertex in $0$.

Although, in general, the operator $\elle$, is not $\delta_r$-homogeneous, nevertheless the following Proposition holds.
\begin{theorem}  Let $\elle$ be a Kolmogorov-Fokker-Planck-type operator as in \eqref{KFP}. 

Let $\Omega$ be a bounded open set of $\erre^{N+1}$ and let be $z_0\in\partial\Omega.$ If there exists a $\delta_r$-cone with vertex in $z_0$ contained in $\erre^{N+1}\backslash \Omega$, then $z_0$ is $\elle$-regular for $\Omega.$ 
\end{theorem} 

\begin{proof} 

As the operator $\elle$ is left translation invariant on $\mathbb{G}$, as in the previous proposition, we prove the theorem in the case $z_0=(0,0)$.
Let  $\hat C$ be a $\delta_r$- cone with vertex in $(0,0)$ such that $\hat C \subseteq \erre^N\backslash \Omega.$  
We denote by $\elle_0$ the {\it principal part} of $\elle$, i.e. the operator 
\begin{eqnarray*}
\elle_0 = \mathrm{div}\left(  A\nabla\right)  +\left\langle B_0 x,\nabla\right\rangle - \partial_t
\end{eqnarray*}
where $B_0$ is the matrix obtained replacing in $B$  all the $\ast$  blocks  by zero matrices. 
$\elle_0$ is hypoelliptic as the matrix
\begin{equation*}
C_0(t)=\int_{0}^{t} E_0(s)AE_0^T(s)\,ds\mbox{,\qquad  }E_0(s) =\exp\left(-sB_0\right),  %
\end{equation*}
is strictly positive  for every $t>0$.  Furthermore,  $\elle$ is left-invariant and homogeneous of degree two 
with respect to the Lie group $(\mathbb {K}_0, \tilde\circ, \delta_r)$ where the composition law is 
$$\left(  x,t\right)  \circ\left(  \xi,\lambda\right)  =\left(  \xi+E_0(\lambda)  x,t+\lambda\right).$$
Then, $\elle_0$ belongs to the class of the operators considered in the previous subsection,  and by Theorem \ref{criterio1}
$$\Gamma_0(0,z) \geq \frac{a}{|C_\lambda (0)|} \qquad \qquad \forall z \in C_\lambda(0), \forall \lambda \in ]0,T[,$$
where $a$ is a suitable real positive constant and $\Gamma_0$ denotes the fundamental solution of $\elle_0$. 
By Theorem 3.1 in \cite{lanconelli_polidoro_1994} there exists a constant $\alpha >0$ such that
$$\Gamma (0,z) \geq \alpha \Gamma_0(0,z) \qquad\qquad\forall z\in \hat C,$$ and then 
$$\Gamma (0,z) \geq  \frac{a \alpha}{|C_\lambda (0)|}.$$
This, thanks to Theorem \eqref{mainth}, proves that $0$ is $\elle$-regular. 
 \end{proof}

\section*{Acknowledgments}

The author has been partially supported by the Gruppo Nazionale per l'Analisi Matematica, la Probabilit\`a e le
loro Applicazioni (GNAMPA) of the Istituto Nazionale di Alta Matematica (INdAM).

\end{document}